\documentclass[11pt]{amsart}
\usepackage{graphicx,color,ulem,overpic,amssymb,hyperref}
\usepackage{graphics}
\usepackage{amsmath}
\usepackage{amscd}
\usepackage{latexsym}
\usepackage[all]{xy}
\begin{document}
\textwidth 5.5in
\textheight 8.3in
\evensidemargin .75in
\oddsidemargin.75in
\newtheorem{quest}{Question}[section]
\newtheorem{thm}[quest]{Theorem}
\newtheorem{lem}[quest]{Lemma}
\newtheorem{defi}[quest]{Definition}
\newtheorem{conj}[quest]{Conjecture}
\newtheorem{cor}[quest]{Corollary}
\newtheorem{prop}[quest]{Proposition}
\newtheorem{prob}[quest]{Problem}
\newtheorem{claim}[quest]{Claim}
\newtheorem{exm}[quest]{Example}
\newtheorem{cond}{Condition}
\newtheorem{rmk}[quest]{Remark}
\newtheorem{que}[quest]{Question}
\newcommand{\p}[3]{\Phi_{p,#1}^{#2}(#3)}
\def\tu{\widetilde{\Upsilon}}
\def\Z{\mathbb Z}
\def\N{\mathbb N}
\def\C{\mathcal{C}}
\def\D{\mathcal{D}}
\def\R{\mathbb R}
\def\g{\overline{g}}
\def\odots{\reflectbox{\text{$\ddots$}}}
\newcommand{\tg}{\overline{g}}
\def\ee{\epsilon_1'}
\def\ef{\epsilon_2'}
\title{
Pochette surgery of 4-sphere}
\author{Tatsumasa Suzuki and Motoo Tange}
\thanks{The first author is supported by JST SPRING, Grant Number JPMJSP2106}
\thanks{The second author is supported by JSPS KAKENHI, Grant Number 21K03216}

\address{Department of Mathematics, Tokyo Institute of Technology,
2-12-1 Ookayama, Meguro-ku, Tokyo 152-8551, Japan}
\email{suzuki.t.do@m.titech.ac.jp}

\address{Institute of Mathematics, University of Tsukuba,
 1-1-1 Tennodai, Tsukuba, Ibaraki 305-8571, Japan}
\email{tange@math.tsukuba.ac.jp}
\date{\today}
\maketitle
\begin{abstract}
Iwase and Matsumoto defined `pochette surgery' as a cut-and-paste on 4-manifolds along a 4-manifold homotopy equivalent to $S^2\vee S^1$. 
The first author in \cite{S2} studied infinitely many homotopy 4-spheres obtained by pochette surgery.
In this paper we compute the homology of pochette surgery of any homology 4-sphere by using `linking number' of a pochette embedding.
We prove that pochette surgery with the trivial cord does not change the diffeomorphism type or gives a Gluck surgery. 
We also show that there exist pochette surgeries on the 4-sphere with a non-trivial core sphere and a non-trivial cord such that the surgeries give the 4-sphere.
\end{abstract}
  \section{Introduction}
  \label{intro}
\subsection{Pochette surgery}
Let $D^n$ be an $n$-dimensional disk and $S^n$ an $n$-dimensional sphere.
Throughout this paper, all manifolds are assumed smooth, connected and oriented, and all maps are smooth. 
For a manifold $M$, the open tubular neighborhood for a submanifold $A\subset M$ is denoted by $N(A)$.
Let $E(X)$ denote the exterior $M-N(X)$ of a submanifold $X$ in $M$.


Here we define pochette surgery, which was initially defined by Iwase and Matsumoto in \cite{IM}.
Let $P$ denote the boundary-sum $S^1\times D^3\natural D^2\times S^2$.
It is called a {\it pochette}.
Let $e$ be an embedding $P\hookrightarrow M$ in a 4-manifold $M$.
Let $Q_e$ denote the image $e(Q)$ of a submanifold $Q$ in $P$.

We take a diffeomorphism $g:\partial P\to \partial E(P_e)$.
We call an operation $M(e,g):=E(P_e)\cup_gP$ a {\it pochette surgery}.
Here we say that the diffeomorphism $g:\partial P\to \partial (E(P_e))$ is a {\it gluing map} for the pochette surgery.

We call the curves $l:=S^1\times \{\text{pt}\}$ and $m:=\partial D^2\times \{\text{pt}\}$ on $\partial P$ a {\it longitude} and a {\it meridian} of $P$, respectively.
According to \cite[Theorem 2]{IM}, the diffeomorphism type of $M(e,g)$ is uniquely determined by the following data:

\begin{itemize}
\item[(i)] an embedding $e: P\hookrightarrow M$,
\item[(ii)] a homology class $g_\ast([m])\in H_1(\partial E(P_e))={\mathbb Z}[m_e]\oplus {\mathbb Z}[l_e]$,
\item[(iii)] a mod $2$ framing $\epsilon$ around $g(m)$.
\end{itemize}
The mod 2 framing will be defined in Section~\ref{mod2f}.
The induced map $g_\ast$ maps the primitive element $[m]$ in $H_1(\partial P)$ to $p[m_e]+q[l_e]$ in $H_1(\partial E(P_e))$, where $p,q$ are relatively prime integers.
Then, we call the element $p/q\in\mathbb{Q}\cup\{\infty\}$ a {\it slope} of the pochette surgery.
Any slope $p/q$ gives an unoriented image $g(m)$ of $m$.
Hence, for some embedding $e$, the pochette surgery with the slope $p/q$ and the mod 2 framing $\epsilon$ is called $(p/q,\epsilon)$-{\it pochette surgery} of $M$ and this denotes $M(e, p/q, \epsilon)$.
We call the 2-sphere $S:=\{\text{pt}\}\times S^2\subset P$ a {\it core sphere} of $P$ and the meridian 2-sphere $B:=\{\text{pt}\}\times \partial D^3\subset P$ a {\it belt sphere} of $P$.

Let $S'$ be an embedded sphere with a product neighborhood in a 4-manifold $M$.
{\it Gluck surgery} along $S'$ is an operation $Gl(S'):=E(S')\cup_\varphi(D^2\times S^2)$, where $\varphi$ is a diffeomorphism $\partial D^2\times S^2\to  \partial N(S')\cong S^1\times S^2$ which is not homotopy equivalent to the identity.
From the construction, any $(\infty,0)$-pochette surgery is the trivial surgery and any $(\infty,1)$-pochette surgery yields $Gl(S_e)$.
In the case of $(0,\epsilon)$-pochette surgery, it is an operation $E(l_e)\cup (D^2\times S^2)$ along the curve $l_e\subset M$.
This surgery means that the result is one side of the manifold obtained by attaching 5-dimensional 2-handle on $M\times I$ along $l_e$.
We call the result an {\it $S^1$-surgery}.
Thus, any pochette surgery with the slope general rational $p/q$ can be regarded as an intermediate between a Gluck surgery and an $S^1$-surgery.

In \cite{M}, Murase studied pochette surgeries of the double of $P$. 
Let $DP$ be the double $P\cup_{\rm id}(-P)=S^1\times S^3\# S^2\times S^2$ of $P$.
Let $i_P$ be the inclusion map $i_P:P\rightarrow DP$.
According to \cite{M}, the manifold $DP(i_P,p/q,\epsilon)$ is diffeomorphic to a Pao manifold \cite{P}. 
In other words, pochette surgery can be used to reduce the Betti numbers. 

Pochette surgery may include interesting 4-manifolds, possibly exotic 4-spheres and so on.
In this article, we focus on pochette surgeries yielding homotopy 4-sphere.

\subsection{Aims of this paper and future works}
A main aim of this paper is to investigate pochette surgeries $M(e,g)$ yielding homotopy 4-spheres and to determine the diffeomorphism types.

First, Theorem~\ref{main} clarifies pochette surgeries $M(e,g)$ of the case where the cord is trivial is diffeomorphic to $M$ or some Gluck surgery.
The cord is defined in the next section.
In Theorem~\ref{main2}, by using this result, we give a sufficient condition that any pochette surgery of $M$ for some core sphere gives the same manifold $M$ or a Gluck surgery.
As a particular condition, any $(1/q,\epsilon)$-pochette surgery of 4-sphere whose core sphere is the unknot is diffeomorphic to $S^4$.
The first author in 
\cite{S2} proved that several examples of infinitely many homotopy $4$-spheres with the trivial core sphere are all diffeomorphic to the standard 4-sphere.
Theorem~\ref{main2} is a generalization of his result.

The second aim is to clarify diffeomorphism types of pochette surgeries in the cases of non-trivial cord.
To do so, we compute the homology of $M(e,g)$ for any homology 4-sphere $M$.
In order to compute the homology, we need to introduce the notion of a linking number for an embedding of a pochette as well as the slope which was defined by Iwase and Matsumoto \cite{IM}.
The linking number of an embedded pochette is the linking number of the embedded core sphere $S_e$ and the longitude $l_e$ in $M$.
It depends on the choice of a meridian $m$, a longitude $l$ and an embedding $e: P\hookrightarrow M$.
Actually, we show that the homology of a pochette surgery is uniquely determined by the slope and the linking number (Proposition~\ref{homology pochette surgery of S4}).

In Section~\ref{examples}, we investigate cases where the core sphere $S_e$ is a non-trivial 2-knot and the cord is a non-trivial (Theorem~\ref{intro:1-fusion}).
These surgeries give the standard 4-sphere.
Actually, we use a ribbon 2-knot of 1-fusion as $S_e$.
The proof is essentially proven in Theorem~\ref{non-trivial knot and cord}.
We generalize this situation to some cases where the core spheres are any general non-trivial ribbon 2-knots $S$ with $\pi_1(E(S_e))\not\cong {\mathbb Z}$ (Theorem~\ref{intro:n-fusion}).
In each case, we give a non-trivial cord such that any surgery yielding homotopy 4-sphere gives the standard $S^4$.
The essential part for this theorem will be proven in Theorem~\ref{non-trivial knot and cord2}.

Pochette surgery can easily extend to a surgery along $\natural^aS^1\times D^3\natural^bD^2\times S^2$ for some positive integers $a,b$.
This is called {\it outer surgery} by \cite{Nak}.
In the future, we expect to find many exotic 4-manifolds by pochette surgery or outer surgery.
See Section~\ref{questions} for questions for pochette surgery or outer surgery.

\subsection{Pochette surgery with the trivial cord.}
Consider $P$ as $D^2\times S^2\cup h^1$, where $h^1$ is a 1-handle.
In order to embed $P$ into a 4-manifold $M$, we have only to determine an embedding of $D^2\times S^2$ and the 1-handle $h^1$.
This gives a properly embedded, simple arc in $E(D^2\times S^2)$ by taking the core of $h^1$.
We call this arc a {\it cord} here.
If a cord is boundary parallel, then the cord is called {\it trivial}.

In this paper, we clarify diffeomorphism types of pochette surgeries of closed 4-manifolds with the trivial cord.
As a corollary, we clarify the diffeomorphism type of any pochette surgery on homology 4-sphere with the complement of the core sphere homotopically trivial.
Here we state the main theorems.

\begin{thm}
\label{main}
Let $e$ be an embedding of $P$ in a closed 4-manifold $M$ with the trivial cord.
Then for any integer $q$, the following holds:
$$M(e,1/q,\epsilon)\cong \begin{cases}M&\epsilon=0\\ Gl(S_e)&\epsilon=1.\end{cases}$$
\end{thm}
This theorem determines diffeomorphism types of $(1/q,\epsilon)$-pochette surgeries with the trivial cord. 
This theorem generalizes previous Okawa's result.
\begin{thm}[Okawa \cite{O}]
\label{okawa}
Let $e$ be an embedding with the cord trivial.
If the core sphere $S_e$ is a ribbon 2-knot, then any pochette surgery $S^4(e,1/q,\epsilon)$ is diffeomorphic to $S^4$ for any integer $q$.
is diffeomorphic to $S^4$ for any integer $q$.
\end{thm}
The Gluck surgery along any ribbon 2-knot is diffeomorphic to the standard 4-sphere, for example, see \cite{GS} about this fact.
Hence, Theorem~\ref{main} implies Theorem~\ref{okawa}.

Theorem \ref{main} immediately leads to the following theorem:
\begin{thm}
\label{main2}
Let $M$ be a homology 4-sphere.
Let $e$ be an embedding $P\hookrightarrow M$ with $\pi_1(E(S_e))={\mathbb Z}$.
If a pochette surgery produces a homology 4-sphere, 
then the result is 
diffeomorphic to $M$ or $Gl(S_e)$.
In particular suppose $M$ is $S^4$ and $e:P\hookrightarrow S^4$ is an embedding that the core sphere $S_e$ is the unknot.
Then if a pochette surgery by $e$ yields a homology 4-sphere $M'$, then $M'$ is diffeomorphic to $S^4$.
\end{thm}

\subsection{Several pochette surgeries with non-trivial core sphere and cord}
Next, we consider several examples of pochette surgeries with non-trivial core sphere and cord.
\begin{thm}
\label{intro:1-fusion}
There exists a pochette embedding $e:P\hookrightarrow S^4$ with a non-trivial core sphere and a non-trivial cord such that 
the pochette surgery $S^4(e,g)$ is diffeomorphic to $S^4$.
\end{thm}
Actually, as proven in Theorem~\ref{non-trivial knot and cord}, the core sphere of $e$ is any non-trivial ribbon 2-knot of 1-fusion.
Furthermore, there exist infinitely many cords for each such ribbon 2-knot such that these results all obtain the standard $S^4$.
\begin{thm}
\label{intro:n-fusion}
For any ribbon 2-knot $S\subset S^4$ with $\pi_1(E(S))\not\cong {\mathbb Z}$, there exists a non-trivial cord $C$ in $E(S)$ such that for the resulting pochette embedding $P\hookrightarrow S^4$ if $S^4(e,g)$ gives a homotopy 4-sphere then it is diffeomorphic to $S^4$.
\end{thm}
In Theorem~\ref{non-trivial knot and cord2} we show that for any ribbon 2-knot there exists a non-trivial cord such that any pochette surgery yielding a homology 4-sphere gives the double of a homology  4-ball without 3-handles.

It is unknown that any pochette surgery with non-trivial $S_e$ gives an exotic manifold or not.
In general, even if $S_e$ is trivial in a 4-manifold $M$, then it is unclear whether the pochette surgery is trivial or not.
We expect that some pochette surgery creates a new exotic 4-manifold.

\subsection*{Acknowledgements}
Hisaaki Endo, who is an adviser of the first author, has constantly encouraged and advised our work and we really thank him.

\section{Preliminaries}
\label{pre}
\subsection{Embedding of \texorpdfstring{$P$}{TEXT}}
\label{embP}
To consider an embedding of $P$ in a 4-manifold $M$, as mentioned in the previous section, we embed a 2-sphere $S$ in $M$ with product neighborhood and embed a cord in the exterior $E(S)$.
In 4-dimension, the isotopy class of any 1-manifold coincides with the homotopy class.
Thus, the isotopy class of any embedding of $P$ is determined by a 2-knot with product neighborhood and the homotopy class of a cord as a properly embedding in $E(S)$.

Let $S$ be a 2-knot in a homology 4-sphere $M$.
Here we clarify the isotopy classes of embedding $e$ of $P$ with $S_e=S$.
We put $G(S)=\pi_1(E(S))$.
$G(S)$ includes a subgroup $\langle m\rangle$ that is isomorphic to ${\mathbb Z}$.
In this section, $m$ is regarded as the class represented by the meridian circle.
Here we call $\langle m\rangle$ a {\it boundary-subgroup}.

In fact, the abelianization map induces the surjection $G(S)\to H_1(E(S))\cong  {\mathbb Z}$ and the meridian is mapped to a generator in ${\mathbb Z}\subset H_1(E(S))$.
Thus $m$ is non-torsion in $G(S)$.
\begin{lem}
Let $S$ be a 2-knot in a homology 4-sphere $M$.
The set of properly embedded cords up to isotopy with the end points included in $\partial E(S)$ has a bijection to the double coset space $G(S)/\!/\langle m\rangle$.
\end{lem}
\begin{proof}
The set of isotopy classes of the cord in $E(S)$ is
$$\Pi_1(E(S),\partial E(S)):=[(I,\partial I),(E(S),\partial E(S))].$$

Let $\pi_1(E(S),\partial E(S))$ be the relative homotopy set.
This set has the following bijection:
$$\pi_1(E(S),\partial E(S))\cong \langle m\rangle\backslash G(S).$$
It is due to the short exact sequence:
$$1\to \pi_1(\partial E(S))\to \pi_1(E(S))=G(S)\to \pi_1(E(S),\partial E(S))\to 1,$$
induced from the homotopy long exact sequence.

On the other hand, the natural map 
$$\varphi:\pi_1(E(S),\partial E(S))\to \Pi_1(E(S),\partial E(S))$$ 
is surjective.
Because any element in $\Pi_1(E(S),\partial E(S))$ can be realized as an element in $\pi_1(E(S),\partial E(S))$ by moving a starting point of the path to the base point $x_0$ of $\pi_1(E(S),\partial E(S))$ by a homotopy.
Suppose that $\varphi(\gamma_0)=\varphi(\gamma_1)$ for some $\gamma_0,\gamma_1\in \pi_1(E(S),\partial E(S))$.
Then $\gamma_0(0)=\gamma_1(0)=x_0\in \partial E(S)$ and $\gamma_0(1),\gamma_1(1)\in \partial E(S)$.
Now there is a homotopy $H:I\times I\to E(S)$ such that $H(0,\cdot)=\gamma_0$, $H(1,\cdot)=\gamma_1$, and $H(t,0),H(t,1)\in \partial E(S)$.
We put $c(t):=H(t,0)$, then $c(t)$ is a loop in $\partial E(S)$ with a base point $x_0$.
Thus $\gamma_0$ and $\gamma_1\cdot c$ are homotopic via a relative homotopy with $x_0$ fixing.
Namely, we have $\gamma_0=\gamma_1\cdot c\in \pi_1(E(S),\partial E(S))$.
Conversely, if $\gamma_0=\gamma_1\cdot c\in \pi_1(E(S),\partial E(S))$ for $c\in \pi_1(\partial E(S))$, then $\gamma_0$ and $\gamma_1$ are the same element in $\Pi_1(E(S),\partial E(S))$.
This means $\pi_1(E(S),\partial E(S))/\langle m\rangle \to \Pi_1(E(S),\partial E(S))$ is bijective.

Thus we obtain the following bijection:
$$\Pi_1(E(S),\partial E(S))\to \pi_1(E(S),\partial E(S))/\langle m\rangle$$
$$\to \langle m\rangle\backslash G(S)/\langle m\rangle=:G(S)/\!/\langle m\rangle.$$
\end{proof}
Let $[[\text{id}]]$ be the element in $G(S)/\!/\langle m\rangle$ represented by the trivial cord.
Here the class in the double coset is represented by $[[\cdot ]]$ and $\text{id}$ stands for the identity element in $G(S)$.
Hence, if the boundary-subgroup $\langle m\rangle$ is a proper subgroup in $G(S)$, then $G(S)/\!/\langle m\rangle\neq \{[[\text{id}]]\}$.
If $S$ is the trivial 2-knot in the 4-sphere, then $G(S)=\langle m\rangle$ and it has a unique isotopy class of a cord.
If $G(S)$ is not isomorphic to ${\mathbb Z}$, then there exists a non-trivial cord.

\subsection{Fundamental group of pochette surgery}
In general, to find a homotopy 4-sphere by pochette surgery, we need to compute the fundamental group.
Let $M$ be a 4-manifold and $e$ an embedding $e:P\hookrightarrow M$.
Let $c_{p,q}$ be the natural lift of $p[m_e]+q[l_e]$ to $\pi_1(\partial E(P_e))$, which is defined in \cite{IM}.
Let $l'$, and $m'$ be the images on $\pi_1(\partial E(P_e))$ of the  based, oriented, longitude and meridian in $\partial P$ via $e$ respectively.
Concretely, the element is given by
$$c_{p,q}=l'^{\lfloor\frac{q}{p}\rfloor}m'l'^{\lfloor\frac{2q}{p}\rfloor-\lfloor\frac{q}{p}\rfloor}m'l'^{\lfloor\frac{3q}{p}\rfloor-\lfloor\frac{2q}{p}\rfloor}\cdots m'l'^{\lfloor\frac{pq}{p}\rfloor-\lfloor\frac{(p-1)q}{p}\rfloor}m'.$$

We assume that the group presentation of $\pi_1(E(S))$ is $\pi_1(E(S))=\langle \mathcal{S}|\mathcal{R}\rangle$, where $\mathcal{S}$ is a set of generators and $\mathcal{R}$ is a set of relators.
For the inclusion maps $i:\partial P_e\to E(P_e)$ and $j:\partial P\to P$, the following maps are induced:
$$i_{\#}:\pi_1(\partial P_e)\to \pi_1(E(P_e)),$$
$$j_{\#}:\pi_1(\partial P)\to \pi_1(P).$$
From the Seifert–Van Kampen theorem, we have
\begin{equation}
\label{piformula}
\pi_1(M(e, p/q, \epsilon))=\langle \mathcal{S}|\mathcal{R},c_{p,q}\rangle.
\end{equation}

\subsection{Mod 2 framing}
\label{mod2f}
We define mod 2 framing of $g(m)$ for a gluing map $g:\partial P\to \partial E(P_e)$ explained in \cite[First paragraph in p.162]{IM}.
Let us consider a pochette surgery on $M$.
After attaching $D^2\times S^2$ in $P$ along $g(m)$, we can uniquely attach the remaining $S^1\times D^3$.
Hence, we have only to consider an identification between neighborhoods of $m$ and $g(m)$ via $g$ to attach $P$.

We fix an identification $\partial P=S^1\times \partial D^3\# \partial D^2\times S^2= S^1\times S^2\#S^1\times S^2$.
The meridian $m=\partial  D^2\times\{\text{pt}\}\subset \partial P$ has the natural product framing.
We obtain an identification $\iota:\partial E(P_e)\to S^1\times S^2\#S^1\times S^2$ through the embedding $e$.
Then, $S^1\times S^2\# S^1\times S^2$ can be presented by the 2-component unlink with 0-framings.
We map the natural framing on $m\subset \partial P$ to a framing on $g(m)$.
The framing is presented by an integer by the identification $\iota$.
As far as we consider the diffeomorphism type of the result of the pochette surgery,
we have only to consider an integer modulo 2 as the framing on $g(m)$.
In fact, consider $P$ as $S^1\times D^3$ attaching a 2-handle with the cocore $m$.
For two gluing maps $g_1, g_2:\partial P\to \partial E(P_e)$ with $g_1(m)=g_2(m)$ but with framings whose difference is divisible by 2,
the map $g_1^{-1}\circ g_2|_{N(m)}$ can be extended to the inside of the 2-handle.
Namely, two 4-manifolds attached by such gluing maps are diffeomorphic each other.
Such a framing on $g(m)$ is called a {\it mod 2 framing} and written by $\epsilon$.
\subsection{Linking number}
\label{linkinnumber}
Let $l$ and $S$ be the longitude and the core sphere of a pochette $P$ respectively.
Let $M$ be an oriented homology 4-sphere and $e:P\hookrightarrow M$ an embedding.
The images $l_e$, and $S_e$ in $M$ give submanifolds of $M$.
Then they can give the linking number: 
$$\ell=L(S_e,l_e),$$
according to \cite{B}.
In fact, we extend an embedding $e|_{S}:S\to M$ to a map $\mathcal{B}^3\to M$, where $\mathcal{B}^3$ is a homology 3-ball.
The orientation of $\mathcal{B}^3$ is induced by the one of $S_e$.
We count the intersection points between the image of $\mathcal{B}^3$ and $l_e$ with sign.
Here we deform $l_e$ in $E(S_e)$ so that $l_e$ can meet with $\mathcal{B}^3$ transversely.
For each intersection point if the concatenation of orientations on $\mathcal{B}^3$ and $l_e$ at the point coincides with the orientation of $M$, then the sign is $+1$, otherwise $-1$.
We call the sign a {\it local intersection number} at the intersection point.
In the end, we sum up the local intersection numbers through all the intersection points.
In the same way, we can compute $L(l_e,S_e)$ by changing the order of $l_e$ and $S_e$.

In the general theory of linking number, the absolute values of $L(S_e,l_e)$ and $L(l_e,S_e)$ are the same.
Actually, by the careful consideration of orientation we can easily obtain $L(S_e,l_e)=-L(l_e,S_e)=:\ell$.
We call this number $\ell$ {\it linking number} of the embedding $e$.
We must notice that the linking number is {\it not} an invariant of the embedding of $P$.
If we fix the coordinate $m$ and $l$, then the linking number can be determined.
This is due to what the 3-disk separating $S^1\times D^3$ and $S^2\times D^2$ is not unique.

Here let us reinterpret the linking number $L(S_e,l_e)$ in terms of the homology.
We use the intersection pairing:
$$\langle \cdot ,\cdot \rangle_3^4:H_3(E(S_e),\partial(E(S_e)))\times H_1(E(S_e))\to {\mathbb Z}.$$
Let $\mathcal{M}^3$ be a Seifert hypersurface of $S_e$ in $E(S_e)$, namely $\mathcal{M}^3$ is a properly embedded 3-manifold in $E(S_e)$ satisfying $\partial \mathcal{M}^3=S_e$.
$H_3(E(S_e),\partial(E(S_e)))$ is isomorphic to ${\mathbb Z}$ generated by $[\mathcal{M}^3]$.
Here $\mathcal{M}^3\cap E(S_e)$ and $\mathcal{M}^3$ are identified.
$H_3(E(P_e),\partial E(P_e))$ is isomorphic to ${\mathbb Z}$ and generated by $[\mathcal{M}^3]$.

The intersection point between $\mathcal{M}^3$ and $m_e$ is one point.
Here we give an orientation on $\mathcal{M}^3$ satisfying $\langle [\mathcal{M}^3],[m_e]\rangle_3^4=+1$.

By the definition of linking number, it follows that $\langle [\mathcal{M}^3],[l_e]\rangle_3^4=\ell$.
Since $H_1(E(S_e))$ is also isomorphic to ${\mathbb Z}$ generated by $[m_e]$, we have $[l_e]=\ell[m_e]$.

In the similar way we consider the next intersection pairing:
$$\langle \cdot,\cdot\rangle_2^4:H_2(E(l_e),\partial(E(l_e)))\times H_2(E(l_e))\to {\mathbb Z}.$$
Here we take a proper embedded surface $\Sigma$ satisfying $\partial \Sigma=l_e$ in $E(l_e)$.
We take the usual orientation of the meridian $B_e$ of $l_e$
and the orientation on $\Sigma$ by using $\langle [\Sigma],[B_e]\rangle_2^4 =+1$.
From the computation $L(l_e,S_e)=-\ell$ of the linking number, we obtain $\langle [\Sigma],[S_e]\rangle_2^4=-\ell$.
Since $H_2(E(l_e))$ is isomorphic to ${\mathbb Z}$ generated by the belt sphere $[B_e]$, $[S_e]=-\ell[B_e]$ holds.

\subsection{The homology of a pochette surgery}
Let $M$ be a homology 4-sphere. 
Here we compute the homology of the result by pochette surgery.
Let $g:\partial P\to \partial E(P_e)$ be a gluing map with the slope $p/q$ and the mod $2$ framing  $\epsilon$.
Let $i$ be the inclusion map $\partial E(P_e)\to E(P_e)$. 

To compute the homology group of any pochette surgery of a homology 4-sphere, we prove lemmas needed later.
First, we compute the homology of $E(P_e)$ here.
Since $E(P_e)$ is connected, we have $H_0(E(P_e))\cong{\mathbb Z}$.
\begin{lem}
\label{homology EPe}
$E(P_e)$ has the following homology groups:
$$H_n(E(P_e))=
\begin{cases}
    {\mathbb Z}\cdot[m_e]&n=1\\
    {\mathbb Z}\cdot[B_e]&n=2\\
0&n\ge 3.
\end{cases}$$
\end{lem}
\begin{proof}
Let $h^3$ be a $4$-dimensional 3-handle.
Attaching $h^3$ on the belt sphere of $P_e$, we obtain 
$E(P_e)\cup h^3=E(S_e)$ and $E(P_e)\cap h^3=\partial D^3\times D^1=S^2\times D^1$. 
The homology of $E(S_e)$ is the same as the homology of $S^1$ and the first homology group is generated by the meridian $m_e$.
Since $H_1$ is independent of attaching any 3-handle, we have
$H_1(E(P_e))=H_1(E(P_e)\cup h^3)=H_1(E(S_e))={\mathbb Z}[m_e]\cong{\mathbb Z}$.
Then we obtain the Mayer-Vietoris sequence:
$$\cdots{\longrightarrow}H_n(S^2\times D^1){\longrightarrow}H_n(E(P_e))\oplus H_n(h^3){\longrightarrow}H_n(E(S_e)){\longrightarrow}\cdots.$$
Thus, we can easily check
$$H_n(E(P_e))=
\begin{cases}
    {\mathbb Z}&n=2\\
    0&n=3,4. 
\end{cases}$$
The generator of $H_1$ clearly corresponds to the meridian $m_e$ of $E(S_e)$ and the one of $H_2$ corresponds to the generator, the belt sphere $B_e$ which is the image of $H_2(S^2\times D^1)$.
\end{proof}

From this lemma, we obtain natural isomorphisms $H_1(E(P_e))\cong H_1(E(S_e))$ and $H_2(E(P_e))\cong H_2(E(l_e))$.
The isomorphisms are induced by the inclusions and connect the corresponding elements $[m_e]$ and $[B_e]$.

Let $g$ be a gluing map from $\partial P$ to $\partial E(P_e)$.
Suppose that $g_\ast([m])=p[m_e]+q[l_e]$ is satisfied on the first homology group.
\begin{lem}
\label{prehomology pochette surgery of S4}
If $g_\ast([m])=p[m_e]+q[l_e]$, then we have $g_\ast([B])=p[B_e]-q[S_e]$.
\end{lem}
\begin{proof}
We put $g_\ast([l])=r[m_e]+s[l_e]$, $g_\ast([B])=x[B_e]+y[S_e]$. 
Then, we can define the non-degenerate bilinear form $\langle\cdot , \cdot \rangle_3: H_1(\partial P)\times H_2(\partial P) \to {\mathbb Z}$ from the cup product $H^2(\partial P)\times H^1(\partial P) \to H^3(\partial P)$.

By defining
$$\langle[m],[B]\rangle_3=0, \langle[l],[B]\rangle_3=1, \langle[m],[S]\rangle_3=1,\text{ and } \langle[l],[S]\rangle_3=0,$$
we determine the orientations on $m$ and $B$.
These orientations coincide with the ones determined Section~\ref{linkinnumber} via the map $H_n(\partial P_e)\to H_n(E(P_e))$.
Since $g:\partial P\to \partial E(P_e)$ is a diffeomorphism map, we can define the non-degenerate bilinear form $\langle\cdot , \cdot \rangle_3^e:H_1(\partial E(P_e))\times H_2(\partial E(P_e)) \to {\mathbb Z}$ from the non-degenerate bilinear form $\langle\cdot , \cdot \rangle_3: H_1(\partial P)\times H_2(\partial P) \to {\mathbb Z}$. 
Since $g:\partial P\to \partial E(P_e)$ is an orientation preserving diffeomorphism map, the determinant of the matrix given by the following presentation 
$$
\left(
\begin{array}{cc}
g_\ast([m]) & g_\ast([l])
\end{array}
\right)
=
\left(
\begin{array}{cc}
[m_e] & [l_e]
\end{array}
\right)
\left(
\begin{array}{cc}
p & r\\
q & s
\end{array}
\right)$$
is 1.
Hence we obtain $ps-qr=1$.
Thus the inverse is as follows:
$$
\left(
\begin{array}{cc}
g_\ast^{-1}([m_e]) & g_\ast^{-1}([l_e])
\end{array}
\right)
=
\left(
\begin{array}{cc}
[m] & [l]
\end{array}
\right)
\left(
\begin{array}{cc}
s & -r\\
-q & p
\end{array}
\right). $$ 
Since
$$\langle g_\ast(\alpha), g_\ast(\beta)\rangle_3^e=\langle\alpha,\beta\rangle_3\mathrm{\ for \ all\ } \alpha\in H_1(\partial P),\ \beta\in H_2(\partial P), $$
we have
\begin{eqnarray*}
x&=&\langle[l_e], x[B_e]+y[S_e]\rangle_3^e \\
 &=&\langle[l_e], g_\ast([B])\rangle_3^e \\
 &=&\langle g_\ast^{-1}([l_e]),[B]\rangle_3 \\
 &=&\langle-r[m]+p[l],[B]\rangle_3 =p 
\end{eqnarray*}
and
\begin{eqnarray*}
y&=&\langle[m_e], x[B_e]+y[S_e]\rangle_3^e \\
 &=&\langle[m_e], g_\ast([B])\rangle_3^e \\
 &=&\langle g_\ast^{-1}([m_e]),[B]\rangle_3 \\
 &=&\langle s[m]-q[l],[B]\rangle_3 =-q.
\end{eqnarray*}
Therefore, we obtain the desired result above.
\end{proof}

\begin{lem}
\label{SB}
Let $e$ be an embedding $P\hookrightarrow M$ with linking number $\ell$.
Let $i$ be an inclusion $i:\partial E(P_e)\to E(P_e)$.
Then $i_*([l_e])=\ell[m_e]$ and $i_*([S_e])=-\ell[B_e]$ are satisfied.
\end{lem}
\begin{proof}
The image of $[l_e]\in H_1(\partial E(P_e))$ by $i_\ast$ is also $[l_e]$ in $H_1(E(P_e))$.
Since $H_1(E(P_e))$ and $H_1(E(S_e))$ are identified with each other by the natural isomorphism by the inclusion, the elements $[m_e]$ having in these homology groups are mapped.
Hence, from Section~\ref{linkinnumber}, $[l_e]=\ell[m_e]$ also holds in $H_1(E(P_e))$.
In the same way, we have $i_\ast([S_e])=-\ell[B_e]$.
\end{proof}
Here, we compute the homology groups of the pochette surgery $M(e, p/q, \epsilon)$.
Since $M$ is connected and oriented, $H_0(M(e, p/q, \epsilon))\cong H_4(M(e, p/q, \epsilon))\cong{\mathbb Z}$ is satisfied.
We compute $H_n$ of $M$ for $n=1,2,3$.

\begin{prop}
\label{homology pochette surgery of S4}
Let $M$ be a homology 4-sphere.
Let $e$ be an embedding with linking number $\ell$.
Then, $M(e, p/q, \epsilon)$ has the following homology groups:
\begin{enumerate}
    \item[(i)] If $p+q\ell\neq 0$, then
$$H_n(M(e,p/q,\epsilon))\cong\begin{cases}
    {\mathbb Z}/(p+q\ell){\mathbb Z}&n=1,2\\
    0&n=3.
    \end{cases}$$
    \item[(ii)]
    If $p+q\ell=0$, then 
    $$H_n(M(e,p/q,\epsilon))\cong\begin{cases}
    {\mathbb Z}&n=1,3\\
    {\mathbb Z}^2&n=2.
    \end{cases}$$
\end{enumerate}
\end{prop}
Note that the case of $p+q\ell=0$ means $(p,q)=(\ell,-1), (-\ell,1)$ because $p,q$ are relatively prime.\\

\begin{proof}
The embedding map $e: P\hookrightarrow M$ induces the following map:
$$H_n(\partial P)\overset{g_{\ast}}{\longrightarrow}H_n(\partial E(P_e))\overset{i_{\ast}}{\longrightarrow}H_n(E(P_e)).$$
Then we have $H_1(\partial E(P_e))={\mathbb Z}\cdot[m_e]\oplus {\mathbb Z}\cdot[l_e]$, $H_2(\partial E(P_e))={\mathbb Z}\cdot[B_e]\oplus {\mathbb Z}\cdot[S_e]$
and obtain $g_{\ast}([m])=p[m_e]+q[l_e]$, $i_{\ast}([m_e])=[m_e]$ and $i_{\ast}([B_e])=[B_e]$.
By Lemma \ref{prehomology pochette surgery of S4}, we obtain $g_{\ast}([B])=p[B_e]-q[S_e]$.
By Lemma \ref{SB}, we have $i_{\ast}([l_e])=\ell[m_e]$ and $i_{\ast}([S_e])=-\ell[B_e]$.
By Lemma \ref{homology EPe} and the Mayer-Vietoris sequence
$$\cdots\to H_\ast(E(P_e))\oplus H_\ast(P)\to H_\ast(M(e,p/q,\epsilon))\to H_{\ast-1}(\partial P)\to \cdots,$$
we obtain the following:\\
$$
  \begin{CD}
      @>>>  0 @>>>  H_3(M(e, p/q, \epsilon))  @>{\partial_{3}}>> {\mathbb Z}\cdot[B]\oplus{\mathbb Z}\cdot[S] \\
      @>{j_{21}\oplus j_{22}}>> {\mathbb Z}\cdot[B_e]\oplus{\mathbb Z}\cdot[S] @>{i_2}>> H_2(M(e, p/q, \epsilon)) @>{\partial_{2}}>> {\mathbb Z}\cdot[m]\oplus{\mathbb Z}\cdot[l] \\
      @>{j_{11}\oplus j_{12}}>> {\mathbb Z}\cdot[m_e]\oplus{\mathbb Z}\cdot[l] @>{i_1}>> H_1(M(e, p/q, \epsilon)) @>{\partial_{1}=0}>> H_0(\partial P).
  \end{CD}
$$
\\
We put $j_n=j_{n1}\oplus j_{n2}$ for any $n\in{\mathbb Z}$.
Since we have $\partial_1=0$, $i_1$ is a surjection.
Since we have 
$j_1([m])=(p+q\ell)[m_e]$ and $j_1([l])=(r+s\ell)[m_e]+[l]$,
we obtain
\begin{eqnarray*}
H_1(M(e, p/q, \epsilon))
&=&\mathrm{Im}\,i_1\\
&\cong&{\mathbb Z}\cdot[m_e]\oplus{\mathbb Z}\cdot[l]/\langle (p+q\ell)[m_e],(r+s\ell)[m_e]+[l]\rangle\\
&\cong&{\mathbb Z}\cdot [m_e]/\langle(p+q\ell)[m_e]\rangle\\
&\cong&{\mathbb Z}/(p+q\ell){\mathbb Z}. 
\end{eqnarray*}
Here $r,s$ are the same coefficients as the ones used in the proof of Lemma~\ref{prehomology pochette surgery of S4}.

Next, we compute $H_2$ and $H_3$ of the result of the pochette surgery.

If $p+q\ell\neq 0$, then $j_1$ is an injection.
Since $i_2$ is a surjection, we obtain the following isomorphism:
\begin{eqnarray*}
H_2(M(e, p/q, \epsilon))
&=&\mathrm{Im}\,i_2\\
&\cong&{\mathbb Z}\cdot[B_e]\oplus{\mathbb Z}\cdot[S]/\langle (p+q\ell)[B_e],(r'+s'\ell)[B_e]+[S]\rangle\\
&\cong&{\mathbb Z}\cdot [B_e]/\langle(p+q\ell)[B_e]\rangle\\
&\cong& {\mathbb Z}/(p+q\ell){\mathbb Z}. 
\end{eqnarray*}
Here, $r',s'$ are some integers satisfying $ps'+qr'=1$.
In this case, $\mathrm{Im}\,\partial_3=\mathrm{Ker}\,j_2=0$.
Thus we have
$$H_3(M(e, p/q, \epsilon))=\mathrm{Ker}\,\partial_3=0.$$

If $p+q\ell=0$, then $\mathrm{Im}\,\partial_2=\mathrm{Ker}\,j_1={\mathbb Z}\cdot[m]$. 
Thus we have
$$H_2(M(e, p/q, \epsilon))\cong\mathrm{Im}\,i_2\oplus{\mathbb Z}\cdot[m]\cong{\mathbb Z}\cdot [B_e]\oplus{\mathbb Z}\cdot[m].$$
In this case, $\mathrm{Im}\,\partial_3=\mathrm{Ker}\,j_2={\mathbb Z}\cdot[B]$. 
Thus we have
$$H_3(M(e, p/q, \epsilon))\cong{\mathbb Z}\cdot [B].$$
Therefore, we obtain the desired result above.
\end{proof}

The theorems by Whitehead in \cite{W}, Freedman in \cite{F} and Lemma~\ref{homology pochette surgery of S4} imply the next corollary.

\begin{cor}
\label{cor}
Let $M$ be a homology 4-sphere. 
$M(e, p/q, \epsilon)$ is homeomorphic to $S^4$ if and only if $M(e, p/q, \epsilon)$ is a simply-connected 4-manifold and $|p+q\ell|$ is equal to 1.
\end{cor}
\begin{proof}
By the Freedman theorem, $M(e, p/q, \epsilon)$ is homeomorphic to $S^4$ if and only if $M(e, p/q, \epsilon)$ is homotopy equivalent to $S^4$. 
We will only show that $M(e, p/q, \epsilon)$ is homotopy equivalent to $S^4$ if and only if $M(e, p/q, \epsilon)$ is a simply-connected 4-manifold and $|p+q\ell|=1$. 
By the Whitehead theorem, the necessary and sufficient condition for a manifold to be homotopy equivalent to $S^4$ is $\pi_1=\{\text{id}\}$ and $H_n=0$ for $n=1,2,3$.
From Proposition~\ref{homology pochette surgery of S4}, we can easily check this corollary follows.
\end{proof}

\subsection{Images of the meridian by diffeomorphism.}
\label{meridianposition}
In this section we describe images of $m$ via some gluing maps $g:\partial P\to \partial E(P_e)$ with slope $1/p$ and $p/(p+1)$.
In the first diagram in Figure~\ref{1qlift} we describe $m,l\subset \#^2S^2\times S^1$.
By sliding along the dashed arrow in the first picture, $m$ is moved to a curve represented by $[m]+[l]$ in the second picture.
Furthermore, sliding the diagram along the dashed arrow, we obtain the third picture.
Then $[m]+[l]$ is moved to a curve by represented by $[m]+2[l]$.
By the same diffeomorphism, $[m]+2[l]$ is moved to a curve represented by $[m]+3[l]$ in the fourth picture.

Thus, by the diffeomorphism $h:\#^2S^2\times S^1\to \#^2S^2\times S^1$ with slope $1/p$, meridian $m$ is moved to a curve represented in $[m]+p[l]$ as in the bottom picture in Figure~\ref{1qlift}.
This position will be used when we describe the handle diagram of $M(e,1/p,\epsilon)$.
    
\begin{figure}[htbp]
\begin{overpic}
{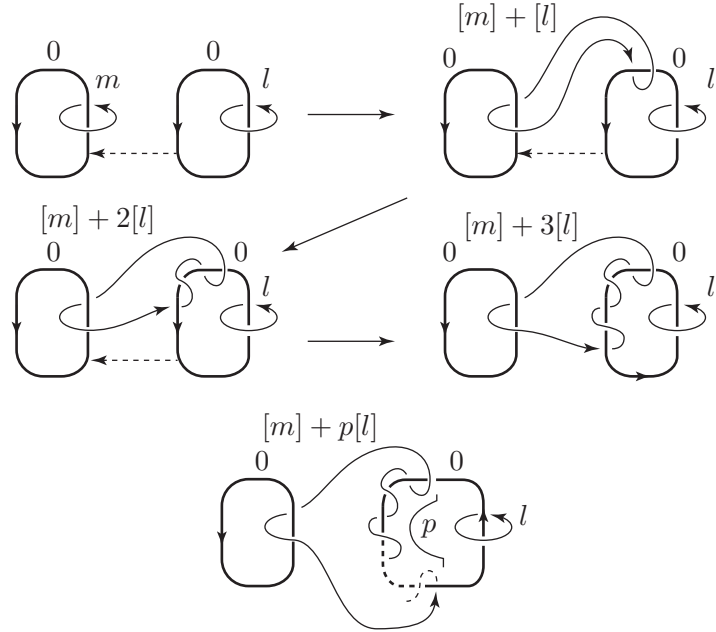}

\put(5,82){0}
\put(28,82){0}
\put(12,78){$m$}
\put(36,78){$l$}

\put (62,81){0}
\put(95,81){0}
\put(64,87){$[m]+[l]$}
\put(100,78){$l$}

\put(5,53){0}
\put(32,53){0}
\put(4,58){$[m]+2[l]$}
\put(36,48){$l$}

\put(62,53){0}
\put(95,53){0}
\put(65,57){$[m]+3[l]$}
\put(100,48){$l$}

\put(35,23){0}
\put(63,23){0}
\put(36,28){$[m]+p[l]$}
\put(73,15){$l$}
\put(59,14){$p$}
\end{overpic}
\caption{Images of $m$ and $l$ via a gluing map $\#^2S^2\times S^1\to \#^2S^2\times S^1$.}
\label{1qlift}
\end{figure}
Furthermore, exchanging $m$ and $l$ in the last picture in Figure~\ref{1qlift} and doing an isotopy, we obtain a curve represented by $p[m]+[l]$ as in the first picture in Figure~\ref{fig:p/p+1} 
and \ref{fig:p/p+12}.
We call these cases Case (I) and Case (II) respectively.
Sliding a 0-framed 2-handle, we obtain the second picture.
The thin curves in the figures are represented by $p[m]+(p+1)[l]$.
By an isotopy we obtain the last pictures in Figure~\ref{fig:p/p+1} and \ref{fig:p/p+12}.
\begin{figure}[htbp]
\begin{overpic}
[
width=.8\textwidth
]
{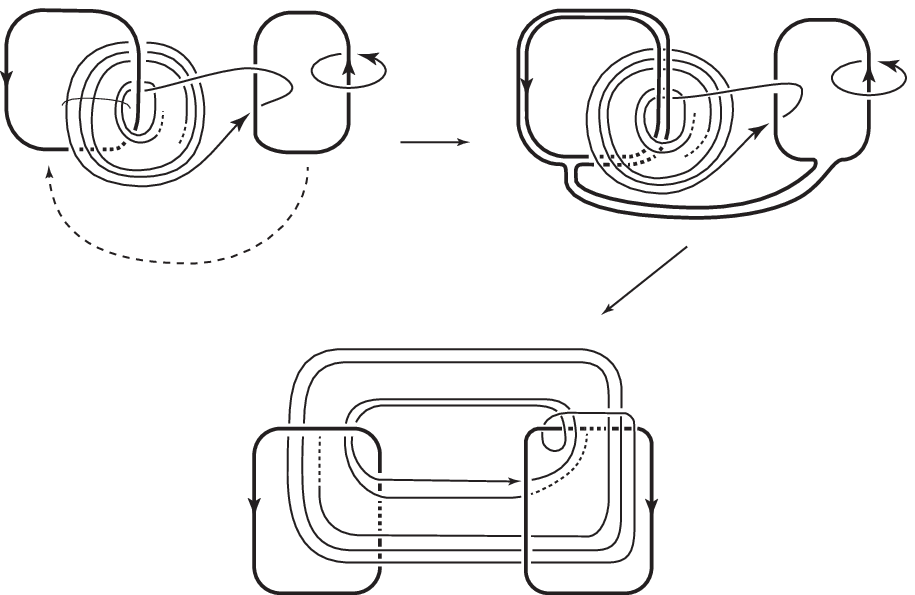}

\put(5,65){0}
\put(4,54){$p$}

\put(60,65){0}
\put(88,65){0}
\put(40,60){$l$}

\put(13,41){$p[m]+[l]$}

\put(33,65){0}
\put(73,8){0}
\put(24,8){0}
\put(97,60){$l$}
\put(70,31){isotopy}
\put(35,29){$p[m]+(p+1)[l]$}

\end{overpic}
    \caption{Case (I).}
    \label{fig:p/p+1}
\end{figure}
\begin{figure}[htbp]
\begin{overpic}
[
width=.8\textwidth
]
{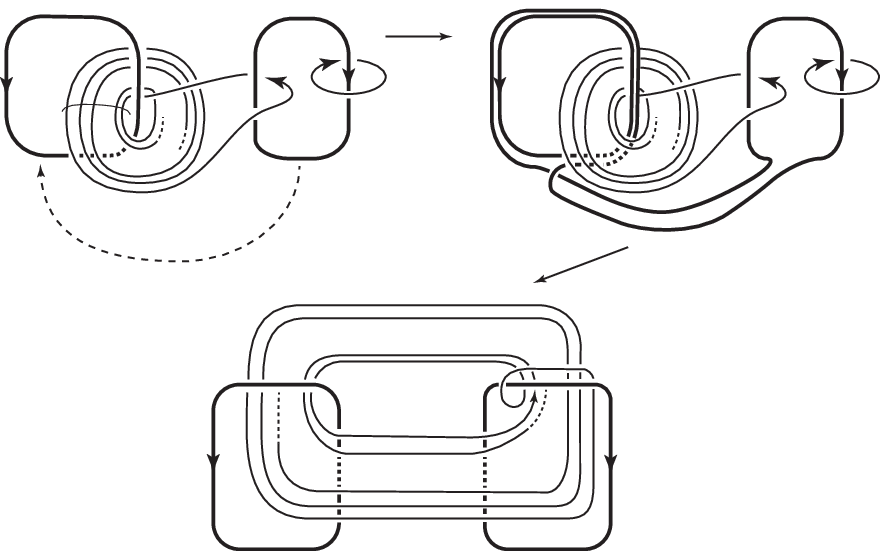}

\put(5,62){0}
\put(4,49){$p$}
\put(13,37){$p[m]+[l]$}

\put(60,62){0}
\put(88,62){0}
\put(41,56){$l$}
\put(30,62){0}

\put(21,8){0}
\put(70,8){0}
\put(97,56){$l$}
\put(68,30){isotopy}
\put(34,30){$p[m]+(p+1)[l]$}
\end{overpic}
    \caption{Case (II).}
    \label{fig:p/p+12}
\end{figure}


\section{Proof of Main theorem}
In this section we prove Theorem~\ref{main}.\\

\begin{proof}[Proof of Theorem~\ref{main}]
Let $e$ be an embedding $P\hookrightarrow M$ with a trivial cord.
The exterior $E(P_e)$ is obtained by attaching a 0-framed 2-handle on $E(S_e)$ in a separated position from the diagram of $E(S_e)$ as in the left picture of Figure~\ref{surgery}.
The circle $m_e$ in the figure is the image of meridian of $P$.
For example, when we describe $E(S_e)$ along the motion picture as in \cite[Section 6.2]{GS}, it is a meridian of a 1-handle corresponding to a 0-handle of the embedded sphere.
Hence, the pochette surgery on $M$ can be obtained by attaching an $\epsilon$-framed 2-handle on $E(P_e)$ plus a 3-handle and a 4-handle.
The position of the $\epsilon$-framed 2-handle is understood from the argument in Section~\ref{meridianposition}.
The right picture in Figure~\ref{surgery} is the local picture of the handle diagram of $M(e,1/q,\epsilon)$.

Here, we prove that the rightmost $0$-framed knot in Figure~\ref{surgery} is isotopic to the unknot in $\partial(E(S_e)\cup h^2(\epsilon))=S^3$,
where $h^2(\epsilon)$ is the $\epsilon$-framed 2-handle.
We remove the previous 3- and 4-handle in $M(e,1/q,\epsilon)$.
Since the obtained manifold is diffeomorphic to the $\epsilon$-Dehn surgery of $\partial E(S_e)$.
By several handle moves of $\partial E(P_e)$, we obtain the Hopf link surgery that the framing coefficients of the two components are $\langle0\rangle$ and $\langle\epsilon\rangle$.
Then we get the second picture in Figure~\ref{unknotproof}.
From this point, doing slides by $q$-times, we obtain the fifth picture.
Canceling the Hopf link component, we obtain $0$-framed knot as in the last picture in Figure~\ref{unknotproof}.
Hence, this 0-framed unknot is isotopic to the unknot.

Since we can move the 0-framed unknot in the last picture in Figure~\ref{surgery} to the unlink position in the same picture, we cancel this component with a 3-handle.
The remaining diagram is obtained by attaching an $\epsilon$-framed 2-handle and a 4-handle on $E(S_e)$.
Therefore, the resulting manifold is the trivial surgery or the Gluck surgery along $S_e$ depending on $\epsilon=0$ or $1$ respectively.
\begin{figure}[htbp]
\begin{overpic}
{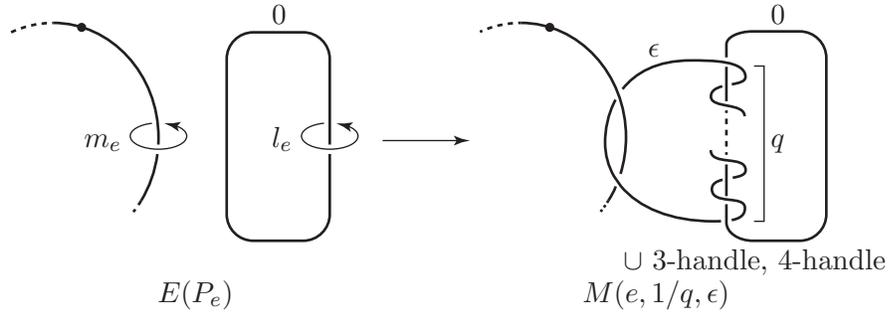}

\put(32,31){$0$}
\put(9,16){$m_e$}
\put(32,16){$l_e$}

\put(18,-3){$E(P_e)$}
\put(70,-3){$M(e,1/q,\epsilon)$}
\put(93,16){$q$}
\put(78,27){$\epsilon$}
\put(93,31){$0$}
\put(75,1){$\cup$ 3-handle, 4-handle}

\end{overpic}
\caption{Attaching $P$ on $E(P_e)$ with the trivial cord.}
\label{surgery}
\end{figure}
\begin{figure}[htbp]
\begin{overpic}
[width=.9\textwidth]
{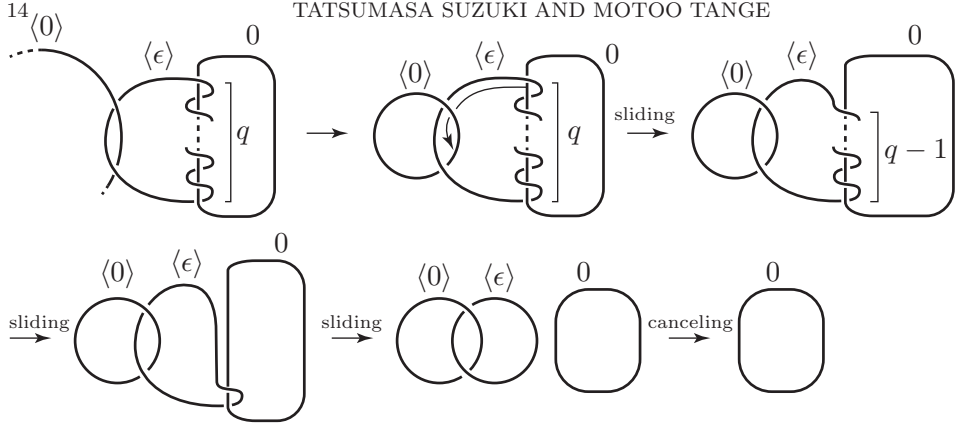}

\put(2,41){$\langle 0\rangle $}
\put(14,38){$\langle\epsilon\rangle$}
\put(24,30){$q$}
\put(25,40){$0$}

\put(41,36){$\langle 0\rangle$}
\put(49,38){$\langle\epsilon\rangle$}
\put(59,30){$q$}
\put(63,38){$0$}

\put(64,32){{\tiny sliding}}

\put(75,36){$\langle 0\rangle$}
\put(82,38){$\langle\epsilon\rangle$}
\put(95,40){$0$}
\put(92.5,28){$q-1$}

\put(0,10){{\tiny sliding}}

\put(9.5,15){$\langle 0\rangle$}
\put(17,16){$\langle \epsilon\rangle$}
\put(28,18){$0$}

\put(33,10){{\tiny sliding}}

\put(43,14.5){$\langle 0\rangle$}
\put(50,14.5){$\langle \epsilon\rangle$}
\put(60,15){$0$}

\put(67.5,10){{\tiny canceling}}

\put(80,15){$0$}

\end{overpic}
\caption{The isotopy type of the rightmost component.}
\label{unknotproof}
\end{figure}
\end{proof}

Using this theorem, we can prove Theorem~\ref{main2}.\\

\begin{proof}[Proof of Theorem~\ref{main2}]
Let $e$ be an embedding $P\hookrightarrow M$.
If $G(S_e)\cong {\mathbb Z}$ holds, then $\pi_1(E(S_e),\partial E(S))$ consists of one element.
This means that any cord in $E(S_e)$ is isotopic to the trivial cord.
Moving the embedded 1-handle in $P$ around the meridian $\partial D^2\times \{\text{pt}\}$ as an isotopy of $e$, we can make the linking number zero.
Hence, if the pochette surgery produces a homology 4-sphere, then the slope is $1/q$ for some meridian and longitude in $P$. 
From Theorem~\ref{main}, the result is $M$ (when $\epsilon=0$) or $Gl(S_e)$ (when $\epsilon=1$).

If $M$ is diffeomorphic to $S^4$ and $S_e$ is the unknot, then any cord is isotopic to the trivial one.
In the same way as above, any pochette surgery yielding a homology 4-sphere gives $S^4$.
\end{proof}

\section{Examples}
\label{examples}
\subsection{Pochette surgeries along ribbon 2-knots of 1-fusion.}
\label{1fusion}
In this section, we consider diffeomorphism types of pochette surgeries on the 4-sphere with non-trivial core spheres and non-trivial cords.
We take any ribbon 2-knot of 1-fusion as core spheres.
Let $S$ denote a ribbon 2-knot of 1-fusion in the 4-sphere.
The sphere $S$ is the double of a disk obtained by attaching one band over two 2-disks as presented by the left picture in Figure~\ref{fig:p/p+1surgery}.
The right diagram is the handle diagram of the complement of $S$.
Let $m'\subset \partial E(S_e)$ be the oriented meridian of a dotted 1-handle indicated in Figure~\ref{fig:p/p+1surgery} with a base point $p$.
Let $l'$ be an oriented meridian of the other dotted 1-handle passing $p$.
Pushing the complement (the dashed line in the right picture in Figure~\ref{fig:p/p+1surgery}) of the neighborhood of $l'$ in the interior of $E(S_e)$, we obtain a cord $C$.
Then the following holds.
\begin{figure}[htpb]
\begin{overpic}
{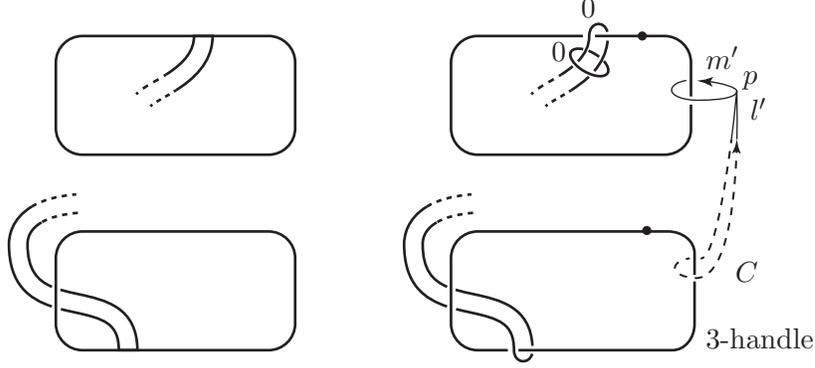}

\put(78,47){$0$}
\put(74,41){$0$}
\put(95,40){$m'$}
\put(99,11){$C$}
\put(101,33){$l'$}
\put(100,38){$p$}
\put(95,2){3-handle}

\end{overpic}
\caption{A ribbon 2-knot of 1-fusion (left) and the diagram of the complement of the 2-knot (right).}
\label{fig:p/p+1surgery}
\end{figure}

\begin{lem}
If $G(S_e)$ is not isomorphic to ${\mathbb Z}$, then this cord $C$ is non-trivial.
\end{lem}
Recall the triviality of a cord was defined in Section~\ref{embP}.\\

\begin{proof}
The fundamental group $G(S_e)$ is presented by 
$$\langle x,y|wxw^{-1}y^{\pm1}\rangle,$$
where $x$ and $y$ are the elements presented by the meridian $m'$ and the longitude $l'$ respectively, and $w$ is a word obtained by reading $x,y$ along the 2-handle corresponding to the band.
Here the boundary-subgroup in $G(S_e)$ is $\langle x\rangle$.

Let $p:G(S_e)\to G(S_e)/\!/\langle x\rangle$ be the projection for the double coset.
Let $[[\text{id}]]$ be the trivial coset in $G(S_e)/\!/\langle x\rangle$, which is the coset including the identity element $\text{id}\in G(S_e)$.
The inverse image $p^{-1}([[\text{id}]])$ is equal to $\langle x\rangle$.
In fact $\langle x\rangle\subset p^{-1}([[\text{id}]])$ is clear.
For any $z\in  p^{-1}([[\text{id}]])$, there exist some integers $r,s$ such that $x^rzx^s=\text{id}$ is satisfied.
Then $z=x^{-r-s}\in \langle x\rangle$.

The homotopy class of the cord $C$ corresponds to $[[y]]\in G(S_e)/\!/{\mathbb Z}$.
If the cord $C$ is trivial, then $y\in p^{-1}([[\text{id}]])=\langle x\rangle$ holds.
Hence we have $y=x^n$ for some integer $n$.
This means $G(S_e)$ is an abelian group.
Since the abelianization of $G(S_e)$ is ${\mathbb Z}$, we have $G(S_e)\cong {\mathbb Z}$.
\end{proof}
In general, it is well-known that $G(S_e)\not\cong {\mathbb Z}$ is satisfied for many non-trivial 2-knot $S_e$.
Then the cord $C$ is non-trivial.

By using this cord $C$, we obtain an embedding $e:P\hookrightarrow S^4$ whose core sphere is $S$.
Then the handle diagram of the complement $E(P_e)$ of $P$ is Figure~\ref{fig:ribbonknot3}.
The meridian $m_e$ is isotopic to $l_1$ or $l_2$ in $E(S_e)$.
Here we assume that $m_e$ is isotopic to $l_i$.
Then, we put the orientation of the longitude as $[l_e]=-[l_i]$ in $E(P_e)$.
Then $[m_e]=-[l_e]$ in $H_1(E(S_e))$ is satisfied.
In this situation, the linking number of $P_e$ is $-1$.
Consider the $(p/(p+1),\epsilon)$-pochette surgery by using the embedding $e$ and these oriented meridian and longitude in $P$.
The element $y\in \pi_1(E(P_e))$ is a lift of $-[l_1]$ and $y^{-1}$ is a lift of $-[l_2]$, hence $y^{\pm1}$ is a lift of the longitude $l_e$.
\begin{lem}
\label{pitriviallemma}
$S^4(e,p/(p+1),\epsilon)$ is simply-connected.
\end{lem}
\begin{proof}
The presentation of $\pi_1(S^4(e,p/(p+1),\epsilon))$ is
$$\langle x,y|wxw^{-1}y^{\pm1},y^{\pm1}(xy^{\pm1}
)^p\rangle$$
according to (\ref{piformula}).
Here we use the rule that the double sign is in the same order.
When we put as $z=xy^{\pm1}$, from the second relation, we obtain $x=z^{p+1}$.
Thus $y^{\pm1}=x^{-1}z=z^{-p}$.
This means the fundamental group above is an abelian group generated by $z$.
Then the first relator $xy^{\pm1}=z^{p+1}z^{-p}=z=\text{id}$ is satisfied.
Therefore, this group is trivial.
\end{proof}
According to the last pictures in Figure~\ref{fig:p/p+1} and \ref{fig:p/p+12} in Section~\ref{meridianposition},
the cases (I) and (II) in Figure~\ref{p/(p+1)surgery2} are obtained as results of attaching $P$ along $p[m_e]+(p+1)[l_e]$ with the mod 2 framing $\epsilon$.
The case (I) is the one which $m_e$ is isotopic to $l_1$ (as an oriented loop), while (II) is the case where $m_e$ is isotopic to $l_2$ in the same way.

\begin{figure}[htpb]
\begin{overpic}
{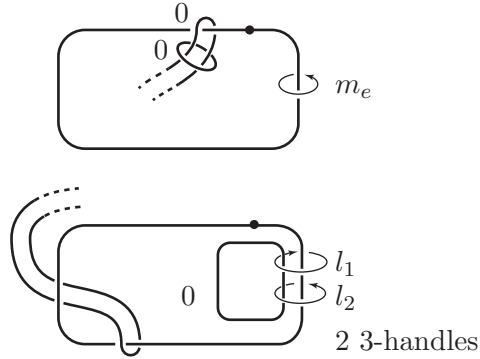}

\put(48,98){$0$}
\put(42,87){$0$}
\put(95,77){$m_e$}
\put(95,25){$l_1$}
\put(95,15){$l_2$}
\put(95,2){2 3-handles}
\put(50,15){$0$}
\end{overpic}
\caption{The pochette complement whose core sphere is a ribbon 2-knot of 1-fusion.}
\label{fig:ribbonknot3}
\end{figure}

\begin{figure}[htpb]
\begin{overpic}
[width=.8\textwidth]
{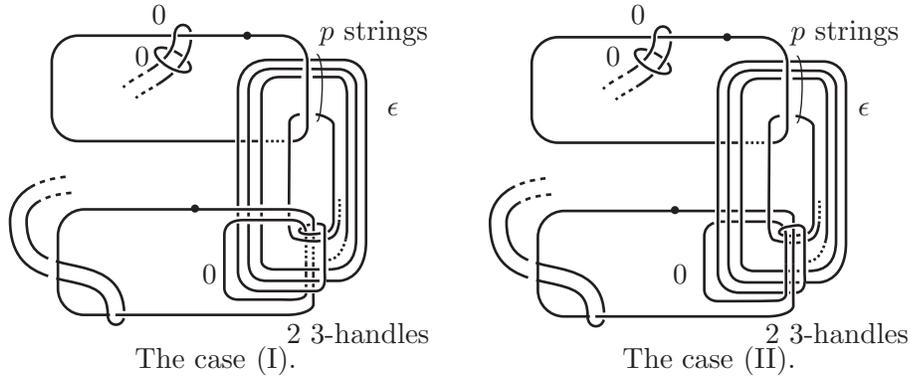}

\put(15,0){The case (I).}
\put(33,3){2 3-handles}

\put(17,41){$0$}
\put(37,39){$p$ strings}
\put(45,30){$\epsilon$}
\put(23,10){$0$}
\put(15,36){$0$}

\put(74,41){$0$}
\put(93,39){$p$ strings}
\put(101,30){$\epsilon$}
\put(79,10){$0$}
\put(71,36){$0$}
    
\put(73,0){The case (II).}

\put(90,3){2 3-handles}
\end{overpic}
\caption{(I): $m_e$ is isotopic to $l_1$, (II): $m_e$ is isotopic to $l_2$.}
\label{p/(p+1)surgery2}
\end{figure}

Lemma~\ref{pitriviallemma} implies that $S^4(e,p/(p+1),\epsilon)$ is a homotopy 4-sphere, but actually the following holds:
\begin{prop}
$S^4(e,p/(p+1),\epsilon)$ is diffeomorphic to the double of a contractible 4-manifold without no 3-handles.
\end{prop}
\begin{proof}
Here we will consider the case where $m_e$ is isotopic to $l_2$.
The case where $m_e$ is isotopic to $l_1$ can be proved in the same way.

We deform the handle diagram of (II) as in Figure~\ref{fig:p/p+1surgery4}.
Continuously, we deform the handle diagram according to Figure~\ref{fig:p/p+1surgery5}.
The last picture presents that $S^4(e,p/(p+1),\epsilon)$ is the double of a contractible 4-manifold $C$.
In fact, the fundamental group of $C$ has the same presentation as the one used in the proof of Lemma~\ref{pitriviallemma}, that is, it is the trivial group.
The homology group of $C$ is easily found out to be trivial from the handle decomposition.
Therefore, $C$ is contractible.
\end{proof}

\begin{figure}[htpb]
\begin{overpic}
{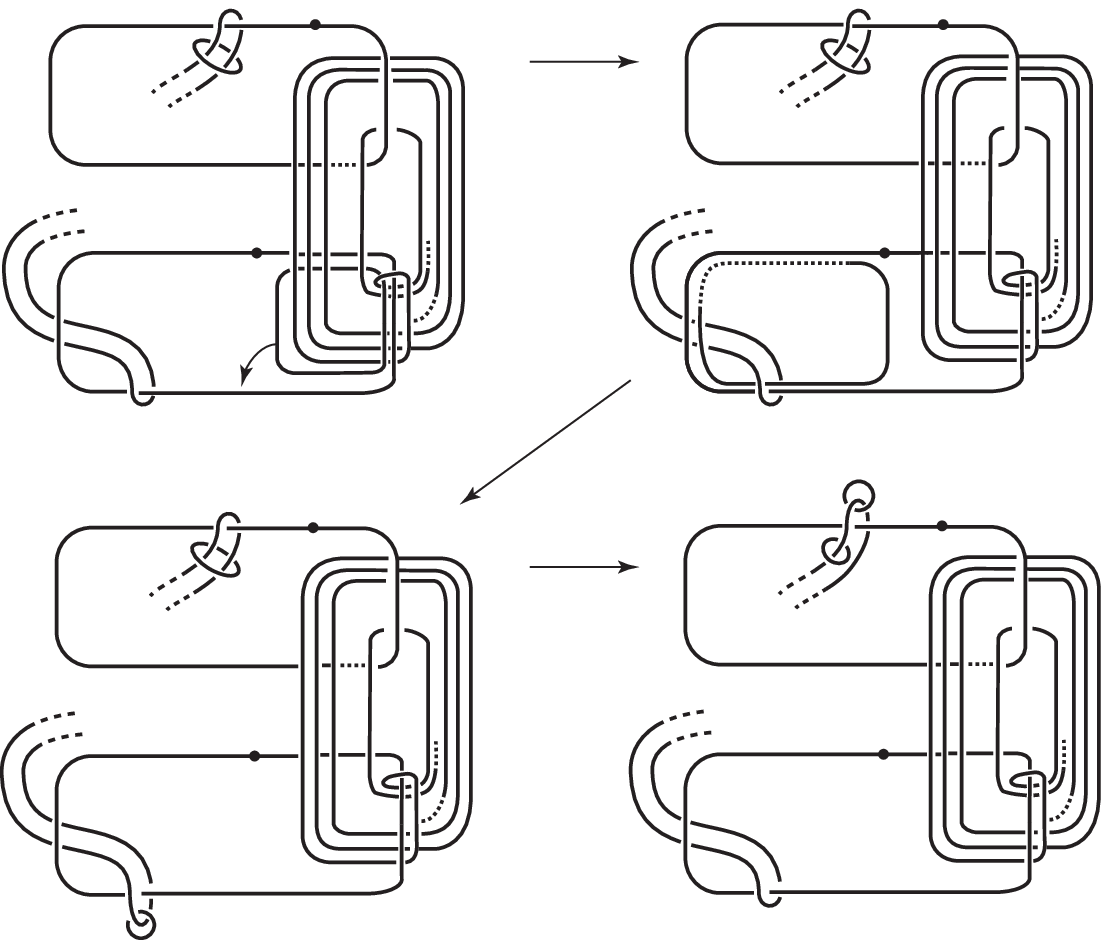}

\put(18,85){0}
\put(15,80){0}
\put(24,72){$\epsilon$}
\put(22,56){0}

\put(43.5,75){handle slide}

\put(74,84){0}
\put(72,80){0}
\put(81,72){$\epsilon$}
\put(77,56){0}

\put(42,52){handle slide}

\put(18,39){0}
\put(15,34){0}
\put(80,40){0}
\put(24,27){$\epsilon$}
\put(9,0){0}

\put(44,36){isotopy \&}
\put(43.5,29){handle slide}

\put(73,39){0}
\put(72,34){0}
\put(81,27){$\epsilon$}

\end{overpic}
\caption{Handle moves.}
\label{fig:p/p+1surgery4}
\end{figure}

\begin{figure}[htpb]
\begin{overpic}
{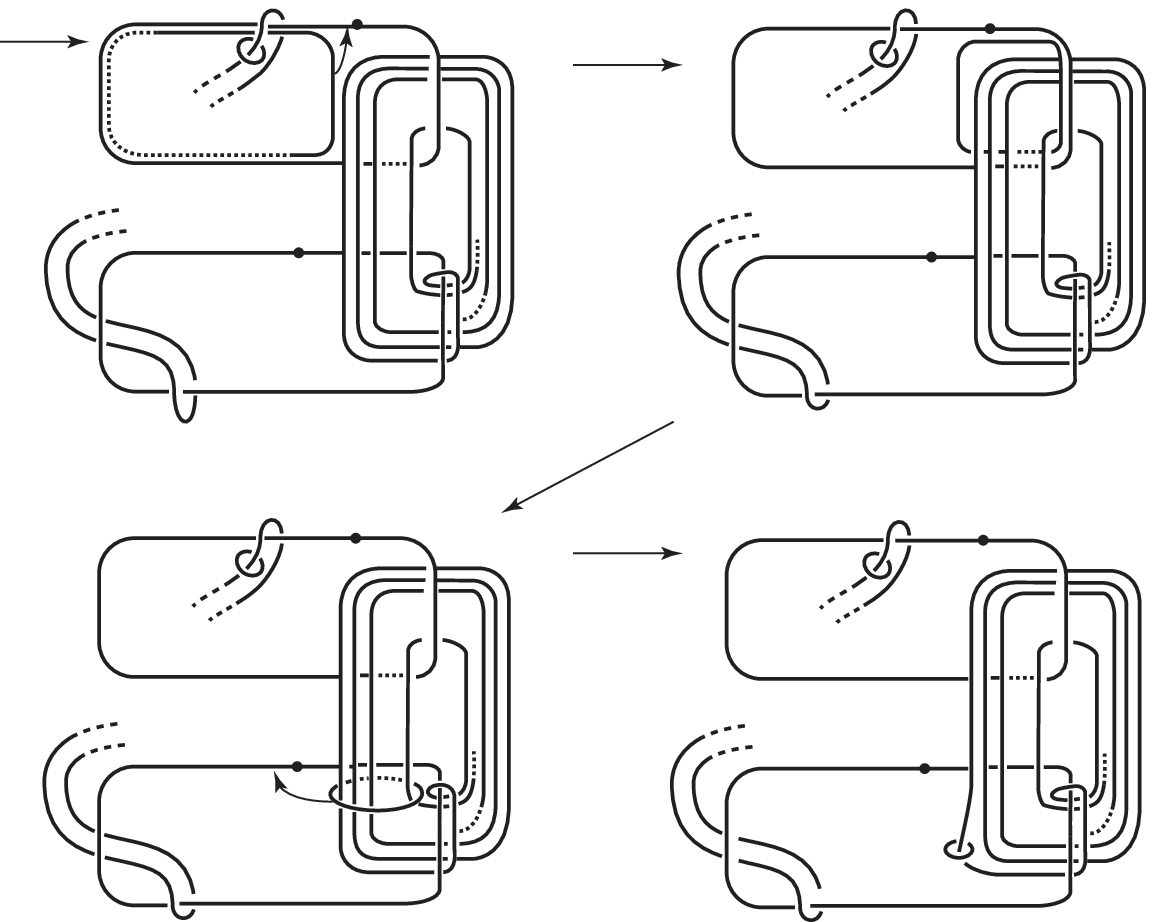}

\put(0,79){handle slide}

\put(20,79){0}
\put(18,74){0}

\put(26,70){0}
\put(27,60){$\epsilon$}

\put(45.5,70){handle slide}

\put(75,79){0}
\put(73,74){0}

\put(80,70){0}
\put(82,60){$\epsilon$}

\put(48,45){isotopy}

\put(20,34){0}
\put(18,30){0}

\put(27,7){0}
\put(27,16){$\epsilon$}

\put(45.5,28){handle slide}
\put(74,34){0}
\put(72,30){0}

\put(79.5,5){0}
\put(82,16){$\epsilon$}

\end{overpic}
\caption{Handle moves.}
\label{fig:p/p+1surgery5}
\end{figure}

As an example, we can give a concrete diagram for the spun trefoil knot as a ribbon 2-knot of 1-fusion.
Figure~\ref{fig:ribbonknot} is the handle diagram of the complement.
\begin{figure}[htbp]
    \centering
    \begin{overpic}[width=.25\textwidth]{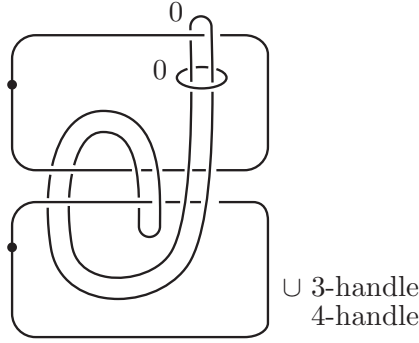}
    \put(45,80){0}
    \put(50,98){0}
\put(85,13){$\cup$ 3-handle}
\put(94,4){4-handle}
\end{overpic}
\caption{A handle diagram of a ribbon 2-knot exterior.}
    \label{fig:ribbonknot}
\end{figure}
We choose $m_e$ and $l_e$ as in the left in Figure~\ref{nontrivial}, then the embedding $i:\partial P_e\hookrightarrow E(P_e)$ gives $i_\ast([l_e])=-[m_e]$.
Namely the linking number is $\ell=-1$.
Let $x,y$ be lifts in $\pi_1(S^4(e,1/2,\epsilon))$ of generators $m_e$ and $l_e$ respectively.
Then the presentation of $\pi_1(S^4(e,1/2,\epsilon))$ is the following: 
$$\langle x,y|yx^{-1}yxy^{-1}x,y^2x\rangle\cong \{\text{id}\}.$$
The diagram of this homotopy 4-sphere becomes the right picture in Figure~\ref{nontrivial}.
In this case, we can deform this diagram into the double of a contractible 4-manifold with no 3-handles as in Figure~\ref{ribboncase2}.
\begin{figure}[htbp]
\begin{overpic}
[width=.7\textwidth]
{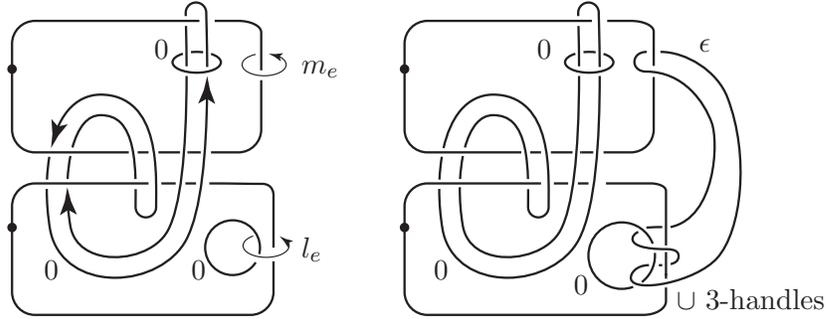}

\put(20,35){$0$}
\put(5,5){$0$}
\put(25 ,5){$0$}
\put(40,33){$m_e$}
\put(40,8){$l_e$}

\put(72,35){$0$}
\put(58,5){$0$}
\put(94,36){$\epsilon$}
\put(77,3){$0$}
\put(91,1){$\cup$ 3-handles}
\end{overpic}
\caption{A pochette surgery $S^4(e,1/2,\epsilon)$ with a non-trivial 2-knot $S_e$ and a non-trivial cord.}
\label{nontrivial}
\end{figure}
\begin{figure}[htbp]
\begin{overpic}
[width=.7\textwidth]
{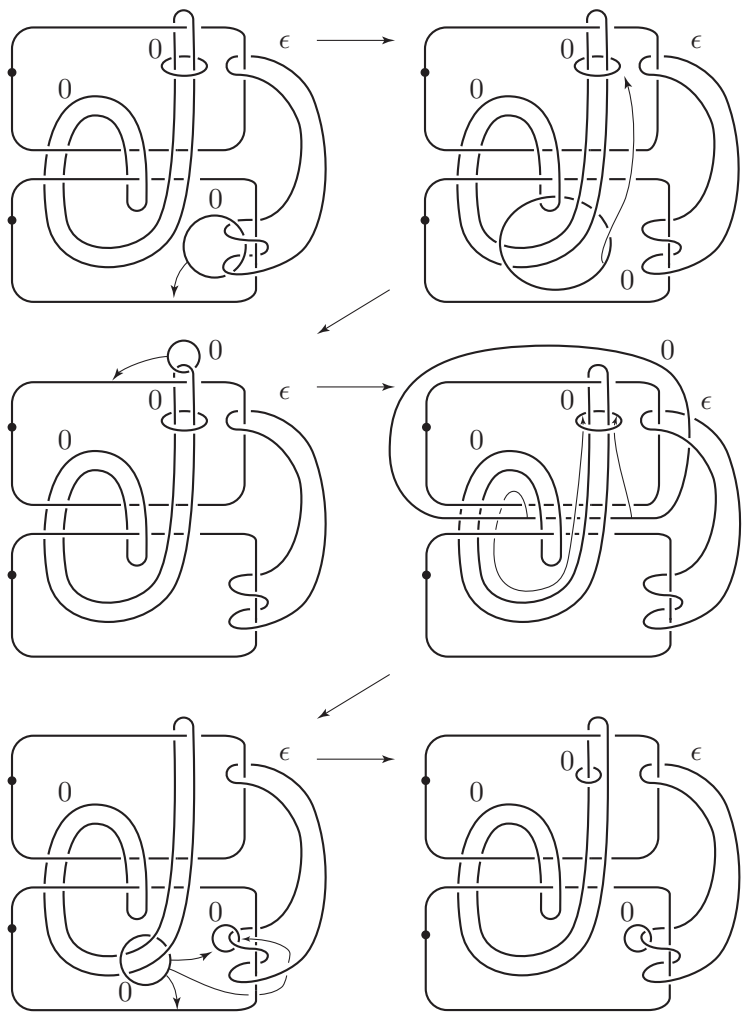}

\put(5,91){$0$}
\put(14,95){$0$}
\put(27,96){$\epsilon$}
\put(20,80){$0$}
 
\put(46,91){$0$}
\put(55,95){$0$}
\put(68,96){$\epsilon$}
\put(61,72){$0$}

\put(5,56){$0$}
\put(14,60){$0$}
\put(27,61){$\epsilon$}
\put(20,65){$0$}

\put(46,56){$0$}
\put(55,60){$0$}
\put(69,60){$\epsilon$}
\put(65,65){$0$}

\put(5,21){$0$}
\put(11,1){$0$}
\put(27,25){$\epsilon$}
\put(20,9){$0$}

\put(46,21){$0$}
\put(55,24){$0$}
\put(68,25){$\epsilon$}
\put(61,9){$0$}
\end{overpic}
\caption{A diffeomorphism to the double of a contractible 4-manifold.}
\label{ribboncase2}
\end{figure}

\begin{thm}
\label{non-trivial knot and cord}
Let $e:P\hookrightarrow S^4$ be any ribbon 2-knot of 1-fusion with $\pi_1(E(S_e))\not\cong {\mathbb Z}$.
Then there exists a non-trivial cord for the pochette such that 
the pochette surgery $S^4(e,p/(p+1),\epsilon)$ is diffeomorphic to $S^4$.
\end{thm}
\begin{proof}
We take the same cord as the previous one used in Figure~\ref{fig:ribbonknot3}.
The double of any contractible 4-manifold $C$ with $2$ 1-handles, $2$ 2-handles and no 3-handles is diffeomorphic to the standard 4-sphere. 
Therefore, $S^4(e,p/(p+1),\epsilon)$ is diffeomorphic to the standard 4-sphere.
\end{proof}
Here we used the method in \cite[the second paragraph in p.\ 36]{A}.
If a contractible 4-manifold $C$ has $n$ 1-handles, $n$ 2-handles and no 3-handles, then the double satisfies $D(C):=C\cup_{\text{id}}(-C)=\partial (C\times I)$.
Since the handle decomposition of $C\times I$ depends only on the homotopy classes of the 2-handles, $C\times I$ is diffeomorphic to the standard $D^5$.

Here we prove Theorem~\ref{intro:1-fusion}.
\begin{proof}[Proof of Theorem~\ref{intro:1-fusion}]
Let $S$ and $C$ be the ribbon 2-knot and the cord that we dealt with in Theorem~\ref{non-trivial knot and cord}.
Then $S$ is non-trivial and $C$ is non-trivial.
The pochette surgery gives the standard $S^4$.
\end{proof}
\subsection{Pochette surgeries along ribbon 2-knots of \texorpdfstring{$n$}{TEXT}-fusion.}
\label{nfusion}
The method to prove Theorem~\ref{non-trivial knot and cord} can be easily extended to the case of the surgery that the core sphere is any ribbon 2-knot of $n$-fusion.
It is unclear whether any result by this pochette surgery is simply-connected or not.
\begin{thm}
\label{non-trivial knot and cord2}
Let $S\subset S^4$ be a ribbon 2-knot with $G(S)\not\cong {\mathbb Z}$.
Then there exists a non-trivial cord $C$ in $E(S)$ satisfying the following conditions:
\begin{enumerate}
\item the core sphere is $S$,
\item the cord is $C$, and 
\item if $S^4(e,g)$ is a homology 4-sphere then it is diffeomorphic to the double of a homology 4-ball without 3-handles.
\end{enumerate}
\end{thm}
\begin{proof}
Let $S$ be any ribbon 2-knot of $n$-fusion.
We fix the handle decomposition of $E(S)$ corresponding to the fusion.
That is, the decomposition has one 0-handle, $n+1$ dotted 1-handles, $n$ 2-handles and $n$ dual 2-handles and $n+1$ 3-handles and one 4-handle.
See \cite[Section 6.2]{GS} for the description of ribbon 2-knot complement.
We take two based meridians $m'$ and $l'$ of the dotted 1-handles with a base point $p_0\in \partial E(S)$.
We suppose that $m'$ lies in $\partial E(S)$ and is a meridian of $\partial E(S)$.
Let $x,y$ be elements in $\pi_1(E(S))$ corresponding to $m'$ and $l'$ respectively.
Here we can assume that $y^{\pm1}$ is conjugate to $x$ but $y^{\pm1}\not\in \langle x\rangle$.
Actually, if any based meridian of each dotted 1-handle of $E(S)$ is in an element in $\langle x\rangle$, then $\pi_1(E(S))$ is a quotient of ${\mathbb Z}$, because the set of the meridians of the dotted 1-handles is a generator of $\pi_1(E(S))$.
Actually using the abelianization map  $\pi_1(E(S))\overset{ab}{\to}H_1(E(S))={\mathbb Z}$, we conclude that $\pi_1(E(S))$ is isomorphic to ${\mathbb Z}$.
Now this case is ruled out.
Thus, there exists a based meridian $l'\subset E(S)$ such that $y:=[l']$ is conjugate to $x$ but $y\not\in \langle x\rangle$.

In the same way as the proof of Theorem~\ref{non-trivial knot and cord}, from $l'$ we produce a cord in $E(S)$.
Thus, by taking such a cord, we obtain a pochette embedding $e:P\hookrightarrow S^4$.
By moving the 0-framed 2-handle by the process in Figure~\ref{fig:p/p+1surgery4} and \ref{fig:p/p+1surgery5}, we can take the 0-framed 2-handle in the position of the meridian of the $\epsilon$-framed 2-handle.

If the graph for the $n$-fusion is as in Figure~\ref{graph1}.
This is just a schematic picture for the fusion, and the edges stand for connecting 0-framed 2-handles coming from the bands of the ribbon disk.
Actually, in the true picture, the edges should be drawn as some bands and might be linking to several dotted 1-handles.   
For our proof, we may omit these data because sliding the 0-framed 2-handle to dual 2-handles, we can ignore the linking.

We take the two based oriented meridians $m'$ and $l'$ in the positions in the figure.
We suppose that the  below 0-framed 2-handle in the first picture in Figure~\ref{fig:p/p+1surgery4} is attached in the dashed circle in Figure~\ref{graph1} in our situation.
From the 1-handle $k$ linking to $l'$ to the 1-handle $k'$ linking to $m'$, the 0-framed 2-handle can be moved by doing several handle slides and some isotopy.
See Figure~\ref{graph2} for the handle moves.
This also generalizes the moves from the first picture in  Figure~\ref{fig:p/p+1surgery4} to the second picture in Figure~\ref{fig:p/p+1surgery5}.
Hence, we can freely move the 0-framed 2-handle from a dotted 1-handle to another dotted 1-handle.
\begin{figure}[htpb]
\begin{overpic}{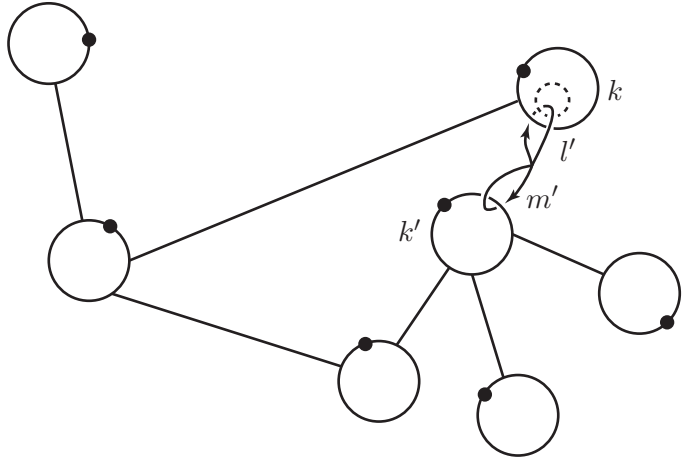}
\put(82,44){$l'$}
\put(77,37){$m'$}
\put(89,53){$k$}
\put(58,32){$k'$}
\end{overpic}
\caption{A graph for the fusion of a ribbon 2-knot.}
\label{graph1}
\end{figure}

\begin{figure}[htpb]
\begin{overpic}{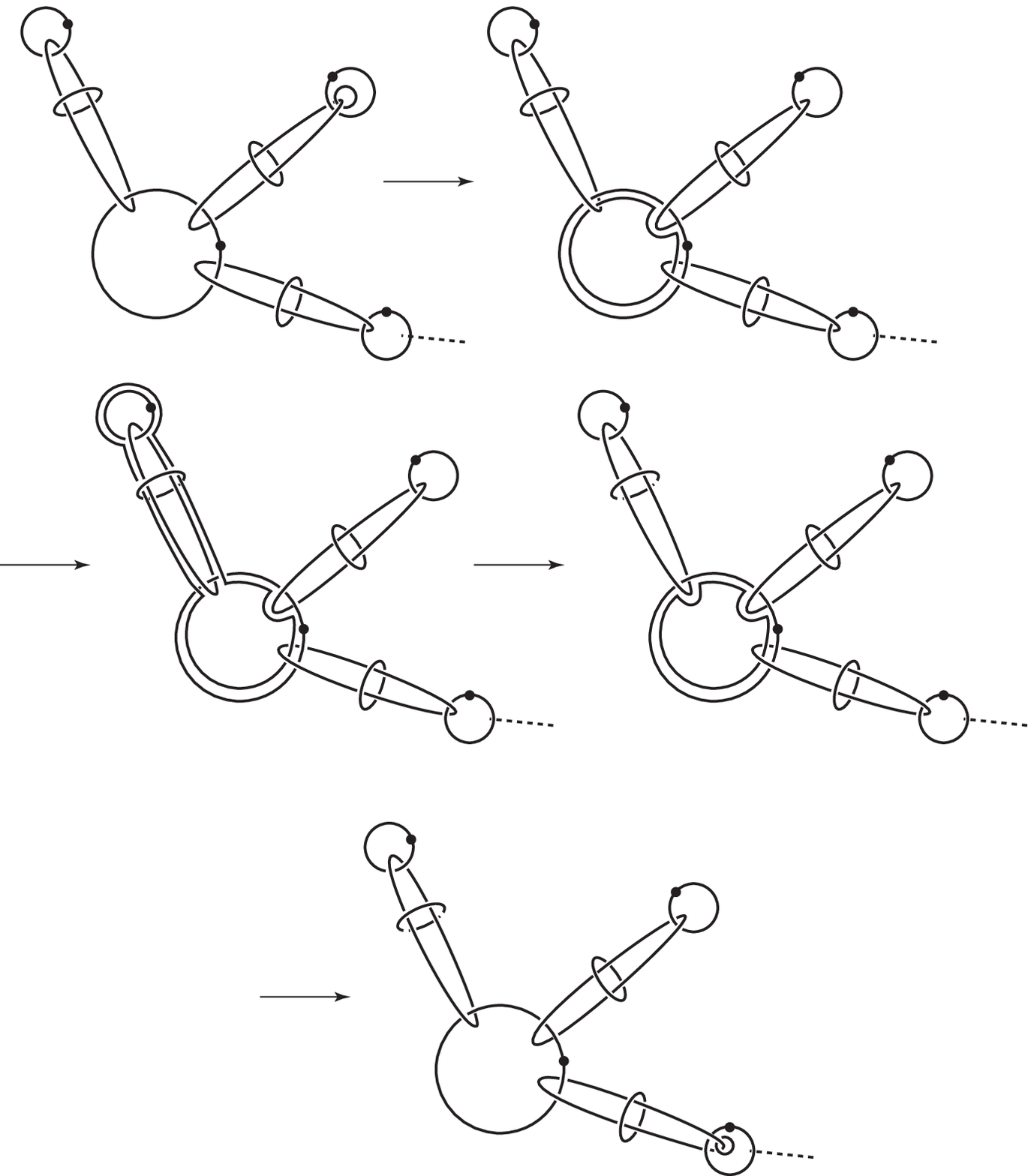}

\put(30,88){$0$}
\put(30,82){handle slide}
\put(0,49){handle slide}
\put(40,49){isotopy}
\put(22,18){isotopy}
\end{overpic}
\caption{Deformations to move the 0-framed meridian in the position.}
\label{graph2}
\end{figure}

By this handle slides, all 0-framed 2-handles corresponding to the dual bands can be moved in the meridians of all 2-handles.
This means that $S^4(e,p/(p+1),\epsilon)$ is the double of a homology 4-ball $H$ without 3-handles.
\end{proof}
If the 2-knot is $n$-fusion ribbon knot, the fundamental group is $S^4(e,p/(p+1),\epsilon)$ has the following form:
$$\langle x_1,\cdots, x_{n+1}|w_1x_{i_1}w^{-1}_1x_{j_1}^{-1},\cdots ,w_nx_{i_n}w_n^{-1}x_{j_n}^{-1},x_s^{-1}(x_rx_s^{-1})^p\rangle,$$
where for $k=1,2,\cdots, n$, $w_k$ is a word in $x_1,\cdots,x_{n+1}$, and the set $\{\{i_k,j_k\}|k=1,\cdots, n\}$ is the set of edges of the graph.
If $H$ in the proof of the previous theorem is a contractible, that is, the fundamental group is trivial, then $S^4(e,p/(p+1),\epsilon)$ is diffeomorphic to $S^4$.
It is unclear whether $H$ is a contractible 4-manifold or not.

Here we prove Theorem~\ref{intro:n-fusion}.
\begin{proof}[Proof of Theorem~\ref{intro:n-fusion}]
Let $S\subset S^4$ be any ribbon 2-knot of $n$-fusion with $G(S)\not\cong {\mathbb Z}$.
Then, we take a non-trivial cord $C$ indicated in Theorem~\ref{non-trivial knot and cord2}.
We obtain a pochette embedding $e:P\hookrightarrow S^4$.
From Theorem~\ref{non-trivial knot and cord2}, if a pochette surgery $S^4(e,g)$ gives a homotopy  4-sphere, then it is a double of a homology 4-ball $B$ without 3-handles.
Here $B$ is a contractible 4-manifold, since $S^4(e,g)$ is simply-connected.
From the method mentioned right after the proof of Theorem~\ref{non-trivial knot and cord}, therefore, $S(e,g)$ is diffeomorphic to the standard $S^4$.
\end{proof}
\section{Questions}
\label{questions}
In this section we raise several questions.
We leave the following problem about Theorem~\ref{non-trivial knot and cord2}.
\begin{que}
Let $S$ be any ribbon 2-knot with $G(S)\not\cong{\mathbb Z}$. 
Does there exist a non-trivial cord $C$ in $E(S)$ such that any nontrivial surgery with respect to the embedding $e:P\hookrightarrow S^4$ with the cord $C$ and the core sphere $S$ yielding a homology 4-sphere gives the standard 4-sphere?
\end{que}

Since pochette surgery is a generalization of Gluck surgery, the triviality of Gluck surgery on any ribbon 2-knot might also hold in the pochette surgery situation.
\begin{que}
Let $S$ be any ribbon 2-knot with $G(S)\not\cong{\mathbb Z}$. 
Suppose that $e:P\hookrightarrow S^4$ is any embedding with $S_e=S$.
Does any pochette surgery $S^4(e,g)$ yielding a homology 4-sphere give the 4-sphere?
\end{que}

Can the diffeomorphisms in the previous section be generalized to cases of any non-trivial core sphere?

\begin{que}
Let $S$ be any 2-knot with $G(S)\not\cong{\mathbb Z}$. 
Then, does there exist a non-trivial cord in $E(S)$ such that any pochette surgery for a pochette embedding $e:P\hookrightarrow S^4$ with the core sphere $S$ is $S^4$ or $Gl(S)$?
\end{que}
Can we construct a homotopy 4-sphere other than $Gl(S)$ by pochette surgery?
Furthermore, we raise two questions in more generalized settings.
\begin{que}
Can a pochette surgery of $S^4$ construct an exotic $S^4$?
\end{que}

Pochette surgery can be generalized to a surgery on a {\it generalized pochette} $P_{a,b}=\natural^a S^1\times D^3\natural^bD^2\times S^2$.
Such a surgery is called an {\it outer surgery} and it is studied by Nakamura in his master thesis \cite{Nak}.
\begin{que}
Can an outer surgery construct an exotic $S^4$?
\end{que}
  
\end{document}